\documentclass[a4paper,12pt,reqno]{amsart}

\title[Zero divisor]{The class of the affine line is a zero divisor in the Grothendieck
ring: via $G _{ 2 }$-Grassmannians}

\author[A.~Ito]{Atsushi Ito}
\address{
Department of Mathematics,
Graduate School of Science,
Kyoto University,
Kyoto 606-8502,
Japan.
}
\email{aito@math.kyoto-u.ac.jp}

\author[M.~Miura]{Makoto Miura}
\address{
Korea Institute for Advanced Study,
85 Hoegiro,
Dongdaemun-gu,
Seoul,
130-722,
Republic of Korea.
}
\email{miura@kias.re.kr}

\author{Shinnosuke Okawa}
\address{
Department of Mathematics,
Graduate School of Science,
Osaka University,
Machikaneyama 1-1,
Toyonaka,
Osaka,
560-0043,
Japan.
}
\email{okawa@math.sci.osaka-u.ac.jp}

\author[K.~Ueda]{Kazushi Ueda}
\address{
Graduate School of Mathematical Sciences,
The University of Tokyo,
3-8-1 Komaba,
Meguro-ku,
Tokyo,
153-8914,
Japan.}
\email{kazushi@ms.u-tokyo.ac.jp}

\date{}
\usepackage[dvipdfm,
bookmarks=true,bookmarksnumbered=true,
 setpagesize=false,%
 colorlinks=true,%
 pdftitle={},%
 pdfsubject={},%
 pdfauthor={},%
 pdfkeywords={TeX; dvipdfmx; hyperref; color;},%
 colorlinks=false]{hyperref}
\usepackage{mymacros,xypic}
\usepackage{oitp}

\begin{document}
\maketitle

\begin{abstract}
Motivated
by \cite{Borisov:2014rt} and \cite{Martin:2016lq},
we show
the equality
$
 \lb [ X ] - [ Y ] \rb \cdot [ \bA ^{ 1 } ] = 0
$
in the Grothendieck ring of varieties,
where
$
 ( X, Y )
$
is a pair of Calabi-Yau 3-folds cut out from the pair of Grassmannians
of type
$
 G _{ 2 }
$.
\end{abstract}


\section{Introduction}

Let
$
 \bfk
$
be an algebraically closed field of characteristic 0. The Grothendieck ring
$
 K _{ 0 } \lb \cV / \bfk \rb
$
of algebraic varieties over
$
 \bfk
$,
as an abelian group, is the free abelian group generated by the set of
isomorphism classes of algebraic varieties over
$
 \bfk
$
modulo the relations
\begin{equation}
 [ X ] = [ X \setminus Z ] + [ Z ],
\end{equation}
where
$
 Z \subset X
$
is a closed subscheme (with the reduced structure).
Products in
$
 K _{ 0 } \lb \cV / \bfk \rb
$
is defined by
\begin{equation}
 [ X ] \cdot [ Y ] = [ X \times Y ].
\end{equation}
As is easily seen, this is associative, commutative, and unital with
$
 1 = [ \Spec \bfk ]
$.

It is
shown in \cite{Borisov:2014rt} that the class
$
 \bL = [ \bA ^{ 1 } ]
$
of the affine line is a zero divisor in
$
 K _{ 0 } \lb \cV / \bfk \rb
$.
It is later refined in \cite{Martin:2016lq} to the formula
\begin{equation}\label{eq:martin}
 \lb [ X ] - [ Y ] \rb \cdot \bL ^{ 6 } = 0,
\end{equation}
where
$
 ( X, Y )
$
is the Pfaffian-Grassmannian pair of Calabi--Yau 3-folds \cite{MR1775415,MR2475813}.
The main result of this paper is the following:

\begin{theorem}\label{th:main_theorem}
Let
$
 ( X, Y )
$
be a pair of Calabi--Yau 3-folds cut out from the pair of
Grassmannians of type
$
 G _{ 2 }
$
as defined in \eqref{eq:definition_of_X_and_Y}. Then
$
 [ X ] \neq [ Y ]
$
and
\begin{equation}\label{eq:main_formula}
 \lb [ X ] - [ Y ] \rb \cdot \bL = 0.
\end{equation}
\end{theorem}

The pair of Calabi--Yau 3-folds
appearing in \pref{th:main_theorem}
is a degeneration of the
Pfaffian-Grassmannian pairs
(see \cite{Inoue-Ito-Miura1} and \cite[Theorem 7.1]{MR3476687}, respectively).
In this sense,
\eqref{eq:main_formula} can be regarded as a stronger result than \eqref{eq:martin}
in the degenerate situation.

The pioneering work
\cite{Borisov:2014rt}
is a counter-example to \cite[Conjecture 2.7]{Galkin:2014zr},
showing that $\bL$ is a zero divisor in
$
 K _{ 0 } \lb \cV / \bfk \rb
$.
It is shown in
\cite{Galkin:2014zr}
that
\cite[Conjecture 2.7]{Galkin:2014zr}
implies the irrationality of the generic cubic 4-folds.

In fact, as pointed out in \cite[Remark 7.2]{Galkin:2014zr},
irrationality of generic $d$-dimensional cubic hypersurfaces would follow from the following slightly weaker version of the conjecture.

\begin{conjecture}\label{cj:galkin_shinder_borisov}
If a class
$
 \alpha
$
is represented by a linear combination of varieties of dimension
less than or equal to
$
 2 ( d - 2 )
$
and satisfies
\begin{equation}
 \alpha \cdot \bL ^{ 2 } = 0,
\end{equation}
then
$
 \alpha
$
belongs to the principal ideal
$
 \la \bL \ra \subset K_0 \lb \cV/\bfk \rb
$
generated by $\bL$.
\end{conjecture}

The next theorem implies that our pairs of
CY 3-folds of type
$
 G _{ 2 }
$
are counter-examples
to \pref{cj:galkin_shinder_borisov}
for
$
 d \ge 4
$.
The authors do not know
if the same holds for the Pfaffian-Grassmannian pairs.

\begin{theorem}
Let
$
 ( X, Y )
$
be as in
\eqref{eq:definition_of_X_and_Y}.
Then
$
 \alpha = [ X ] - [ Y ]
 \not \in
 \la \bL \ra
$.
\end{theorem}

\begin{proof}
Suppose otherwise.
Then since there exists the well-defined ring isomorphism
\begin{equation}
 K _{ 0 } \lb \cV / \bfk \rb / \la \bL \ra \simto \bZ [ \mathrm{SB} ]
\end{equation}
which sends the class of a smooth projective variety to its stable birational class,
it follows that
$
 X
$
and
$
 Y
$
are stably birational (see \cite[Section 1.7 and Theorem 2.3]{MR1996804}).
By the uniqueness of the base space of the MRC fibration
\cite[(2.9) Corollary]{MR1158625}
(see also \cite[Chapter IV Section 5]{MR1440180}),
it follows that
$
 X
$
and
$
 Y
$
are birational to each other. But this is a contradiction, as explained in the proof of
\pref{th:main_theorem}.
\end{proof}

As is the case for the Pfaffian-Grassmannian pairs
\cite[Theorem 2.13]{Borisov:2014rt},
the pairs of Calabi--Yau 3-folds of type
$
 G _{ 2 }
$
also give a negative answer to \cite[Question 1.2]{MR1996804},
which asks whether any two varieties with equal classes
in the Grothendieck ring can be cut up
into isomorphic pieces.


\begin{acknowledgements}
This joint work started with a discussion between M.~M.~and S.~O.~at Korea Institute for Advanced Study.
S.~O.~is indebted to Seung-Jo Jung for the invitation.
The authors also thank Kiwamu Watanabe for answering their question
on homogeneous varieties, Sergey Galkin for various useful comments, and
Grzegorz Kapustka and Micha\l \  Kapustka for pointing out the reference
\cite{MR3476687}.

A.~I.~was supported by the Grant-in-Aid for JSPS fellows, No.\ 26--1881.
A part of this work was done when M.~M.~was supported by Frontiers of
Mathematical Sciences and Physics at University of Tokyo. 
M.~M.~was also supported by Korea Institute for Advanced Study.
S.~O.~was partially supported by Grants-in-Aid for Scientific Research
(16H05994,
16K13746,
16H02141,
16K13743,
16K13755,
16H06337)
and the Inamori Foundation.
K.~U.~was partially supported by Grants-in-Aid for Scientific Research
(24740043,
15KT0105,
16K13743,
16H03930).

\end{acknowledgements}

%
%
\section{The example}

Let $G$ be the simply-connected simple algebraic group
associated with the Dynkin diagram \Fzero \ 
of type
$
 G _{ 2 }.
$
The long root and the short root will be denoted by
$\alpha_1$ and $\alpha_2$ respectively.
The Weyl group is generated by reflections
$s_1$ and $s_2$
along the roots $\alpha_1$ and $\alpha_2$,
and isomorphic to the hexagonal dihedral group;
\begin{equation}\label{eq:weyl_group_of_G2}
 W
 =
 \la  s _{ 1 }, s _{ 2 } \relmid
 s _{ 1 } ^{ 2 } = s _{ 2 } ^{ 2 } = \lb s _{ 1 } s _{ 2 } \rb ^{ 6 } = e \ra.
\end{equation}
The parabolic subgroups of $G$
associated with the crossed Dynkin diagrams
\Fone, \Ftwo \ and \Fthree \  will be denoted by
$P_1$, $P_2$ and $B$ respectively.
The corresponding homogeneous varieties
form the following diagram:
\begin{align}\label{eq:the_hat}
 \xymatrix{
 &
 F _{ 1 2 } : = G / B
 \ar _{ p _{ 1 } }[dl]
 \ar ^{ p _{ 2 } }[dr]
 &\\
 F _{ 1 } : = G / P _{ 1 }
 &
 &
 F _{ 2 } : = G / P _{ 2 }}
\end{align}

\noindent
The dimensions of
$F_1$, $F_2$ and $F_{12}$ are
are 5, 5 and 6 respectively.
The morphisms
$
 p _{ 1 }
$
and
$
 p _{ 2 }
$
in
\eqref{eq:the_hat}
are both
$
 \bP ^{ 1 }
$-bundles associated to locally free sheaves of rank 2.
To see this, note that homogeneous varieties are rational and
hence their (cohomological) Brauer groups are trivial.

The Picard groups of
$
 F _{ 1 }
$
and
$
 F _{ 2 }
$
are isomorphic to
$
 \bZ
$.
The ample generators
$
 \cO _{ F _{ i } } ( 1 )
$
are in fact very ample (see e.g. \cite[Theorem 6.5 (2)]{MR1021290}).
The Picard group of
$
 F _{ 1 2 }
$
satisfies
\begin{equation}
 \Pic \lb F _{ 1 2 } \rb
 =
 p _{ 1 } ^{ * } \Pic \lb F _{ 1 } \rb
 \oplus
 p _{ 2 } ^{ * } \Pic \lb F _{ 2 } \rb,
\end{equation}
and the line bundle
$
 \cE _{ ( 1, 1 ) } = \cO ( 1 ) \boxtimes \cO ( 1 )
$
is very ample. Moreover it has degree 1 on the fibers of
$
 p _{ 1 }
$
and
$
 p _{ 2 }
$,
since one can check that
$
 p _{ i * } \cE _{ ( 1, 1 ) }
$
is a locally free sheaf of rank 2 by directly computing the corresponding representations of the Levi subgroups and using that
the corresponding highest weight is preserved under pushforwards (see e.g. \cite[p.48]{MR1038279}). Thus we conclude that
$
 F _{ 1 2 }
$
is the projective bundle over
$
 F _{ i }
$
associated to
$
 p _{ i * } \cE _{ ( 1, 1 ) }
$.

Since
$
 \cE _{ ( 1, 1 ) }
$
is globally generated, so are
$
 p _{ i * } \cE _{ ( 1, 1 ) }
$
by
\cite[Proposition 1.8]{MR1201388}.
Let
\begin{equation}
 s \in H ^{ 0 } \lb F _{ 1 2 }, \cE _{ ( 1, 1 ) } \rb
 \simeq
 H ^{ 0 } \lb F _{ 1 }, p _{ 1 * } \cE _{ ( 1, 1 ) } \rb
 \simeq
 H ^{ 0 } \lb F _{ 2 }, p _{ 2 * } \cE _{ ( 1, 1 ) } \rb
\end{equation}
be a generic global section and set
\begin{equation}\label{eq:definition_of_X_and_Y}
\begin{split}
 D \coloneqq Z ( s ) \subset F _{ 1 2 },\\
 X \coloneqq Z ( p _{ 1 * } s ) \subset F _{ 1 },\\
 Y \coloneqq Z ( p _{ 2 * } s ) \subset F _{ 2 }.
\end{split}
\end{equation}
Because of the genericity of $s$,
the zero loci
$
 X, Y,
$
and
$
 D
$
are all smooth over the base field
\cite[Theorem 1.10]{MR1201388}.
Moreover,
one has
$
 \deg \cO _{ X } ( 1 ) = 42 \ne 14 = \deg \cO _{ Y } ( 1 ),
$
where
$
 \cO _{ X } ( 1 ) \coloneqq \cO _{ F _{ 1 } } ( 1 ) | _{ X }
$
and
$
 \cO _{ Y } ( 1 ) \coloneqq \cO _{ F _{ 2 } } ( 1 ) | _{ Y }
$
are generators of
$\Pic X$ and $\Pic Y$ respectively
(see e.g.~\cite{G2_CY}).

%

Set
$
 \pi _{ i } = p _{ i } | _{ D } \colon D \to F _{ i }
$
for
$
 i = 1, 2
$.

\begin{lemma}
\label{lm:structure_of_the_birational_maps}
The morphism
$
 \pi _{ 1 } \colon D \rightarrow F_1
$
coincides with the blowing up of $F_1$ along $X$.
In particular,
$
 \pi _{ 1 }
$
is an isomorphism outside of
$
 X \subset F _{ 1 }
$
and is a
$
 \bP ^{ 1 }
$-bundle over
$
 X
$
in the Zariski topology.
The same holds for
$
 \pi _{ 2 }
$.
\end{lemma}

\begin{proof}
Since the arguments for
$
 \pi _{ 1 }
$
and
$
 \pi _{ 2 }
$
are the same, we only give it for
$
 \pi _{ 1 }
$.
For simplicity of notation,
set $\cE_1 =  p _{ 1 * } \cE _{ ( 1, 1 ) } $ and $s_1 = p _{ 1 * } s  \in H^0(F_1,\cE_1)$.
Let $I_X \subset \cO_{F_1} $ be the ideal sheaf of $X$.
Since $s_1$ is generic,
it is locally identified with a regular sequence of length two
and hence the codimension of the zero locus $X =Z(s_1)$ in $F_1$ coincides with the expected one.
Thus there exists an exact sequence
\begin{equation}\label{eq_koszul}
0 \rightarrow \cO_{F_1} \rightarrow \cE_1 \rightarrow I_X \cdot \det \cE_1 \rightarrow 0,
\end{equation}
from which we obtain another exact sequence
\begin{equation}\label{eq_blow_up}
0 \rightarrow \bigoplus_{k \in \bN} \Sym^{k-1} \cE_1
\stackrel{\cdot s_1}{\longrightarrow}
\bigoplus_{k \in \bN} \Sym^k \cE_1
\longrightarrow
\bigoplus_{k \in \bN} I^k_X \cdot (\det \cE_1)^{\otimes k} \rightarrow 0,
\end{equation}
where we set $\Sym^{-1} \cE_1 =0$.
Under the identification of  $F_{12 } $ with $\bP_{F_1} ( \cE_1) =  \bfProj_{F_1} \bigoplus \Sym^k \cE_1$,
$D$ is the zero locus of $s_1 \in H^0(\bP_{F_1} ( \cE_1), \cO_{\bP_{F_1} ( \cE_1)}(1))$.
Hence
we have $D=\bfProj_{F_1} \bigoplus I^k_X \cdot (\det \cE_1)^{\otimes k} $ by the exact sequence \eqref{eq_blow_up}.
Since the blowing up $\bfProj_{F_1} \bigoplus I^k_X $ is isomorphic to $\bfProj_{F_1} \bigoplus I^k_X \cdot (\det \cE_1)^{\otimes k} $,
this lemma follows.
\end{proof}

\begin{proposition}\label{pr:F1_and_F2_have_the_same_motive}
$
 [ F _{ 1 } ] = [ F _{ 2 } ]
$.
\end{proposition}

\begin{proof}
This follows from the Bruhat decomposition 
for rational homogeneous varieties
(see e.g.~\cite[Proposition 1.3]{MR1092142})
and the fact that
there is a bijection
$
 b \colon
 W ^{ P _{ 1 } }
 \xrightarrow[]{\simeq}
 W ^{ P _{ 2 } }
$
which respects the lengths of the elements.
Here
\begin{equation}
 W ^{ P _{ i } }
 =
 \lc w \in W |
 w \ \mbox{is the shortest element in} \ w W _{ P _{ i } }
 \rc,
\end{equation}
where
$
 W _{ P _{ i } }
$
is the Weyl group of
$
 P _{ i }
$.

More explicitly,
the subgroup
$
 W _{ P _{ i } }
$
is identified with
$
 \la s _{ i } \ra
 =
 \lc e, s _{ i } \rc
$
and the subsets
$
 W ^{ P _{ 1 } }
$
and
$
 W ^{ P _{ 2 } }
$
are identified with
\begin{equation}\label{eq:WP1}
 \lc e, s _{ 2 }, s _{ 1 } s _{ 2 }, s _{ 2 } s _{ 1 } s _{ 2 }, s _{ 1 } s _{ 2 } s _{ 1 } s _{ 2 },
 s _{ 2 } s _{ 1 } s _{ 2 } s _{ 1 } s _{ 2 }\rc
\end{equation}
and
\begin{equation}\label{eq:WP2}
 \lc e, s _{ 1 }, s _{ 2 } s _{ 1 }, s _{ 1 } s _{ 2 } s _{ 1 }, s _{ 2 } s _{ 1 } s _{ 2 } s _{ 1 },
 s _{ 1 } s _{ 2 } s _{ 1 } s _{ 2 } s _{ 1 } \rc,
\end{equation}
respectively.
The bijection
$
 b
$
sends the
$
 j
$-th member of
\eqref{eq:WP1} to that of
\eqref{eq:WP2}.
With these descriptions, we can in fact show
\begin{equation}
 [ F _{ 1 } ] =
 [ F _{ 2 } ] =
 1 + \bL + \bL ^{ 2 } + \bL ^{ 3 } + \bL ^{ 4 } + \bL ^{ 5 } =
 [ \bP ^{ 5 } ].
\end{equation}
\end{proof}

\begin{proof}[Proof of \pref{th:main_theorem}]
Since
$
 X
$
and
$
 Y
$
are non-isomorphic
Calabi--Yau 3-folds of Picard number one, they are not birationally equivalent.
Hence are not stably birational by the uniqueness of the MRC fibration, so that
$
 [ X ] \neq [ Y ]
$.

In order to show
\eqref{eq:main_formula}, note that
\pref{lm:structure_of_the_birational_maps} implies
\begin{equation}
 \begin{split}
 [ D ]
 =
 [ D \setminus \pi _{ 1 } ^{ - 1 } ( X ) ]
 +
 [ \pi _{ 1 } ^{ - 1 } ( X ) ]
 =
 [ F _{ 1 } \setminus X ]
 +
 \lb 1 + \bL \rb \cdot [ X ]
 =
 [ F _{ 1 } ] + \bL \cdot [ X ]
 \end{split}
\end{equation}
and similarly
\begin{equation}
  [ D ]
 =
 [ F _{ 2 } ] + \bL \cdot [ Y ].
\end{equation}
Now the conclusion follows from
\pref{pr:F1_and_F2_have_the_same_motive}.
\end{proof}

\bibliographystyle{amsalpha}
\bibliography{annihilator}

\end{document}